\documentclass[12pt, reqno]{amsart}
\usepackage{amsmath, amsthm, amscd, amsfonts, amssymb,color,mathrsfs,}

\usepackage[left=2.5cm, right=2.5cm, top=2.5cm, bottom=2.5cm]{geometry}
\usepackage{hyperref}
\hypersetup{
    colorlinks=true,
    linkcolor=blue,
    filecolor=darkmagenta,      
    urlcolor=cyan,
    citecolor=red
} 
\usepackage{amsthm}
\usepackage{graphicx}
\newtheorem{theorem}{Theorem}[section]
\newtheorem{lemma}[theorem]{Lemma}
\newtheorem{proposition}[theorem]{Proposition}
\newtheorem{corollary}[theorem]{Corollary}
\newtheorem{conjecture}[theorem]{Conjecture}
\theoremstyle{definition}
\newtheorem{definition}[theorem]{Definition}
\newtheorem{question}[theorem]{Question}
\newtheorem{example}[theorem]{Example}
\newtheorem{remark}[theorem]{Remark}

\numberwithin{equation}{section}

\begin{document}

\author{Pankaj Dey and Mithun Mukherjee}

\address{Pankaj Dey, School of Mathematics, Indian Institute of Science Education and Research, Thiruvananthapuram, Maruthamala PO, Vithura, Thiruvananthapuram-695551, Kerala, India.}
\email{\textcolor[rgb]{0.00,0.00,0.84}{pankajdey15@iisertvm.ac.in}}

\address{Mithun Mukherjee, School of Mathematical \& Computational Sciences, Indian Association for the Cultivation of Science, 2A \& 2B Raja S C Mullick Road, Jadavpur, Kolkata-700032, INDIA.}
\email{\textcolor[rgb]{0.00,0.00,0.84}{mithun.mukherjee@iacs.res.in}}

\title{Higher Rank Numerical Ranges of Normal Operators and unitary dilations}

\begin{abstract}
We describe here the higher rank numerical range, as defined by {Choi, Kribs and \.{Z}yczkowski}, of a normal operator on an infinite dimensional Hilbert space in terms of its spectral measure. This generalizes a result of Avenda\~{n}o for self-adjoint operators. An analogous description of the numerical range of a normal operator by Durszt is derived for the higher rank numerical range as an immediate consequence. It has several interesting applications. We show using Durszt's example that there exists a normal contraction $T$ for which the intersection of the higher rank numerical ranges of all unitary dilations of $T$ contains the higher rank numerical range of $T$ as a proper subset. Finally, we strengthen and generalize a result of Wu by providing a necessary and sufficient condition for the higher rank numerical range of a normal contraction being equal to the intersection of the higher rank numerical ranges of all possible unitary dilations of it.   
\end{abstract}
\maketitle

\let\thefootnote\relax

\section{Introduction}
The notion of the quadratic forms associated with matrix was extended for operators on Hilbert space which is known as the numerical range. Let $\mathscr{H}$ be a Hilbert space and $\mathscr{B(\mathscr{H})}$ be the algebra of bounded linear maps on $\mathscr{H}$. Suppose $T\in\mathscr{B\mathscr{(H)}}$. The numerical range of $T$ denoted by $W(T)$ is defined as
$$W(T):=\left\{\langle Tf,f\rangle: f\in \mathscr{H},\Vert f\Vert=1\right\}.$$
It was studied extensively. A celebrated result in the theory of the numerical range is that $W(T)$ is a convex set which is known as Toeplitz Hausdorff Theorem (cf. \cite{GR}). Durszt  \cite{D} showed that the numerical range of a normal operator $T$ is the intersection of all convex Borel set $s$ in $\mathbb{C}$ such that $E(s)=I$ where $E$ is the unique spectral measure associated with $T$ and used it to settle a conjecture of Halmos \cite{PH2} in negative.

Choi, Kribs and \.{Z}yczkowski have first defined the higher rank numerical range in the context of ``quantum error correction" \cite{CKZ} which is defined in the following way. Let $T\in\mathscr{B\mathscr{(H)}}$ and $k\in\mathbb{N}_{\infty}:=\mathbb{N}\cup\{\infty\}$. The $k$-rank numerical range of $T$ denoted by $\Lambda_k(T)$ is defined as
$$\Lambda_k(T):=\left\{\lambda\in\mathbb{C}: PTP=\lambda P,\text{ for some projection } P \text{ of rank }k\right\}$$
or, equivalently $\lambda\in\Lambda_k(T)$ if and only if there is an orthonormal set $\{f_j\}_{j=1}^k$ such that $\langle Tf_j,f_k\rangle=\lambda\delta_{j,k}$ for $j,k\in\{1,2,\cdots,k\}$. Clearly, 
$$W(T)=\Lambda_1(T)\supseteq\Lambda_2(T)\supseteq\cdots\supseteq\Lambda_k(T)\supseteq\cdots.$$

Choi, Kribs and \.{Z}yczkowski have given a description of the higher rank numerical range of self-adjoint matrices in terms of its eigenvalues \cite{CKZ}. Let $M_n$ be the algebra of $n$-by-$n$ matrices with complex entries. Suppose $T\in M_n$ be a self-adjoint matrix and $1\leq k\leq n$. CKZ showed
\begin{align}
\Lambda_k(T)=\left[\lambda_{n-k+1},\lambda_k\right]\label{hermitian matrix}
\end{align}
where $\lambda_1\geq\lambda_2\geq\cdots\geq\lambda_n$ are eigenvalues of $T$. In the same article, they put forward a conjecture (CKZ conjecture) on the higher rank numerical range of normal matrices. Let $T\in M_n$ be a normal matrix and $1\leq k\leq n$. CKZ conjecture says
\begin{align}
\Lambda_k(T)=\bigcap\limits_{1\leq j_1<\cdots<j_{n-k+1}\leq n} \text{ conv }\left\{\lambda_{j_1},\ldots,\lambda_{j_{n-k+1}}\right\}\label{CKZ conjecture}
\end{align}
where $\lambda_1,\ldots\lambda_n$ are eigenvalues of $T$. Condition for the higher rank numerical range of a matrix being non-empty was also studied in \cite{CKZ} and \cite{LPS2}.

Li, Poon and Sze have described the higher rank numerical range of a matrix as the intersection of the closed half plane determined by eigenvalue which settled CKZ conjecture in affirmative \cite{LS}. Let $T\in M_n$ and $1\leq k\leq n$. They showed
\begin{align}
\Lambda_k(T)=\bigcap\limits_{\xi\in[0,2\pi)}\left\{\mu\in\mathbb{C}:\Re({e^{i\xi}\mu})\leq\lambda_k(\Re({e^{i\xi}T}))\right\}.\label{geometric description}
\end{align}
It also proved that the higher rank numerical range of a matrix is always convex. However, the convexity of the higher rank numerical range of any operator was first independently shown by Woerdeman \cite{HW} (see also \cite{CGHK}).

There are some work on extending the finite dimensional results on the higher rank numerical range to infinite dimension. Avenda\~{n}o has described the higher rank numerical range of self-adjoint operators \cite{RM} on infinite dimensional Hilbert space in terms of the spectral measure associated with it which extends (\ref{hermitian matrix}). Let $T\in\mathscr{B(\mathscr{H})}$ be self-adjoint operator and $E$ be the unique spectral measure associated with $T$. Suppose $k\in\mathbb{N}_{\infty}$. Let
\begin{align*}
&A_k:=\left\{a\in\mathbb{R}:\text{ dim ran}E(-\infty,a]<k\right\},\\
& B_k:=\left\{b\in\mathbb{R}:\text{ dim ran}E[b,\infty)<k\right\},\\
&\widetilde{\Omega}_k:=A_k^c\cap B_k^c.
\end{align*}
Avenda\~{n}o proved the following.
\begin{theorem}[Theorem 3.4, \cite{RM}]\label{result for selfadjoint operator}
Let $T\in\mathscr{B(\mathscr{H})}$ be self-adjoint operator and $k\in\mathbb{N}_{\infty}$. Then
\begin{align*}
\Lambda_k(T)=\bigcap\limits_{V\in\mathscr{V}_k} W(V^*TV)=\widetilde{\Omega}_k
\end{align*}
where $\mathscr{V}_k$ be the set of all isometries $V:\mathscr{H}\rightarrow\mathscr{H}$ such that codimension of ran $V$ is less than $k$.
\end{theorem}
In other words, the above theorem states that $\lambda\in\Lambda_k(T)$ if and only if $\text{dim ran}E(-\infty,\lambda]\geq k$ and $\text{dim ran}E[\lambda,\infty)\geq k$ for $k\in\mathbb{N}_{\infty}$ and self-adjoint $T$. We also mention here the fact that the $k$-rank numerical range of an operator on infinite dimensional Hilbert space is always non-empty for $k\in\mathbb{N},$ obtained as a corollary of Theorem \ref{result for selfadjoint operator} (Corollary 3.5, \cite{RM}). There is also an independent proof of the above fact in \cite{LPS2}. However, $\infty$-rank numerical range of an operator could be empty (cf. \cite{LPS}, \cite{RM}).

Li, Poon and Sze extended (\ref{geometric description}) for any operator. Suppose $T\in\mathscr{B\mathscr{(H)}}$ and $k\in\mathbb{N}$. Let
$$V_k(T):=\bigcap\limits_{\xi\in[0,2\pi)}\left\{\mu\in\mathbb{C}:\Re({e^{i\xi}\mu})\leq\lambda_k(\Re({e^{i\xi}T}))\right\}$$
where $$\lambda_k(H):=\sup\left\{\lambda_k(V^*HV): V:\mathbb{C}^k\rightarrow\mathscr{H}\text{ such that } V^*V=I\right\}$$ for some selfadjoint operator $H\in\mathscr{B(\mathscr{H})}$. Let $\text{\textbf{Int}}(S)$ denote relative interior of $S$ and $\overline{S}$ denote the closure of $S$ where $S\subseteq\mathbb{C}$. Li and etc. showed the following.

\begin{theorem}[Theorem 2.1, \cite{LPS}]\label{geomertic description for operator}
Let $T\in\mathscr{B(\mathscr{H})}$ and $k\in\mathbb{N}$. Then
$$\text{\textbf{Int}}\left(V_k(T)\right)\subseteq\Lambda_k(T)\subseteq V_k(T)=\overline{\Lambda_k(T)}.$$
\end{theorem}

In the same article, Li and etc. have also obtained that $\Lambda_{\infty}(T)=\bigcap\limits_{k\geq 1}\Lambda_k(T)$ (Theorem 5.1, \cite{LPS}) for $T\in\mathscr{B(\mathscr{H})}$. For more details on the higher rank numerical range of operators on infinite dimensional Hilbert space, the reader may look at \cite{HG}, \cite{LPS}, \cite{RM}.

In this paper, we describe the higher rank numerical range of a normal operator on an infinite dimensional Hilbert space in terms of the spectral measure associated with it (Theorem \ref{main theorem}) which extends (\ref{CKZ conjecture}). An analogous description of the numerical range of normal operators by Durszt \cite{D} is derived for the higher rank numerical ranges as an immediate corollary (Corollary \ref{analogous description}). It has several other applications. (Corollary \ref{compute numerical range}, Corollary \ref{necessary condition for infinity numerical range being empty}). An elementary alternative proof of Theorem 4.4 \cite{LPS} is also given (Corollary \ref{alternative proof}). The description of the higher rank numerical range of the bilateral shift is immediate using the main theorem of this paper (Example \ref{bilateral shift}).

As an application of the main theorem of this paper, we show using Durszt's \cite{D} example that there exists a normal contraction $T$ for which the higher rank numerical range of all unitary dilations of $T$ contains the higher rank numerical range of $T$ as a proper subset (Theorem \ref{generalising Durszt's example}). We strengthen and generalize a result of Wu \cite{W} by providing a necessary and sufficient condition on a normal contraction $T$ that its higher rank numerical range equals the intersection of the higher rank numerical ranges of unitary dilations of $T$ (Theorem \ref{generalizing Wu's theorem}). Finally, we propose a necessary and sufficient condition on a contraction $T$ acting on an infinite dimensional Hilbert space that its higher rank numerical range equals the intersection of the higher rank numerical ranges of unitary dilations of $T$ (Conjecture \ref{NASC for Halmos conjecture to be true}).

\section{Preliminaries}

We, throughout this paper, denote by $\mathscr{H}$, an infinite dimensional separable Hilbert space and $\mathscr{B(\mathscr{H})}$, the algebra of all bounded linear maps on $\mathscr{H}.$ Set $\mathbb{N}_{\infty}:=\mathbb{N}\cup\{\infty\}$. For $T\in\mathscr{B(\mathscr{H})}$, we denote by $\sigma_e(T)$, the eigen-spectrum of $T$ and $\sigma(T)$, the spectrum of $T$.  Let us denote $\mathbb{D}:=\{z\in\mathbb{C}:\vert z\vert<1\}$ and $\mathbb{T}:=\{z\in\mathbb{C}:\vert z\vert=1\}$. We will also continue to use $S^\circ$, the topological interior of $S$ and $\delta S$, the boundary of $S$ for $S\subseteq\mathbb{C}$.

Let $k\in\mathbb{N}_{\infty}$. If the dimension of a subspace is less than $k$ then we mean that the dimension of that subspace is at most $k-1$ whenever $k$ is finite and same is finite whenever $k$ is infinite.

Let us start with few basic lemmas.

\begin{lemma}\label{partition of disjoint set}
Let $X$ be a topological space and $S\subseteq X$. Then $\overline{S}=S^\circ\cup\delta S.$
\end{lemma}
\begin{proof}
Let $x\in\overline{S}$ such that $x\notin S^\circ$. Suppose $U$ be an open set containing $x$. As $x\in\overline{S}$, we have $S\cap U\neq\emptyset$. Now, we claim that $U\cap S^c\neq\emptyset$. If possible let $U\cap S^c=\emptyset$. Then $U\subseteq S$ which contradicts that $x\notin S^\circ$. So, $U\cap S^c\neq\emptyset$. Hence $x\in\delta S$. The other containment is immediate.
\end{proof}

\begin{lemma}\label{separating a convex set}
Let $S$ be a convex set in $\mathbb{C}$ and $\lambda\notin S$. Then there exists a closed half plane $H$ with the line $\delta H$ passing through $\lambda$ such that $H$ contains $S$.
\end{lemma}
\begin{proof}
It is immediate using separation theorem of convex set by half plane.
\end{proof}

\begin{lemma}\label{closure interior of a convex set}
Let $S$ be a convex set in $\mathbb{C}$. Then $(\overline{S})^\circ=S^\circ$.
\end{lemma}

\begin{proof}
Let $\lambda\in(\overline{S})^\circ$. Then there exists an open ball $B$ containing $\lambda$ such that $B\subseteq \overline{S}$. Let $\lambda\notin S^\circ$. Then by Lemma \ref{partition of disjoint set}, $\lambda\in\delta S$. So, $B\cap S^c\neq\emptyset$. Let $\mu\in B\cap S^c$. As $\mu\notin S$ and $S$ is a convex set, there exists a closed half plane $H$ with the line $\delta H$ passing through $\mu$ such that $S\subseteq H$ by Lemma \ref{separating a convex set}. Again $\mu\in(\overline{S})^\circ$. Then there exists an open ball $C$ containing $\mu$ such that $C\subseteq\overline{S}$. Take $D=C\cap H^c$. Note that $D$ is non empty (as $\mu$ lies on $\delta H$), open set and contains boundary point(s) of $S$ as $D=C\cap H^c\subseteq \overline{S}\cap H^c=\delta S\cap H^c$. So, $D\cap S\neq\emptyset$ which is a contradiction. Hence $\lambda\in S^\circ$. The other containment is trivial.
\end{proof}

\begin{lemma}[Lemma, \cite{D}]\label{positive measure}
Let $S$ be a convex set in the complex plane and $\mu$ be a non-negative Borel measure defined on $S$ such that $\mu(S)=1$. Then $\int\limits_{S}zd\mu\in S$.
\end{lemma}

\begin{lemma}\label{SOT limit of projections}
Let $k\in\mathbb{N}$ and $\{P_n\}_{n=1}^\infty$ be a sequence of projection with $\text{dim ran}P_n<k$ for all $n$ such that $P_n\xrightarrow{\text{SOT}}P$ where $P$ is a projection. Then $\text{dim ran}P<k$.
\end{lemma}
\begin{proof}
Let $\text{dim ran}P\geq k$. Suppose $\{e_r:1\leq r\leq k\}$ be an orthonormal set in $\text{ran}P$. Let $A_n=\left(\langle P_n e_j,e_i\rangle\right)_{i,j=1}^k\in M_k$. As $P_n\xrightarrow{\text{SOT}}P$, we have $\lim\limits_{n\to\infty}P_n e_r=P e_r=e_r$ for all $r\in\{1,\cdots,k\}$. So, $A_n\rightarrow I_k$ in norm. Since the set of all invertible matrices is open, there exists $n_{\circ}\in\mathbb{N}$ such that $A_{n_{\circ}}$ is invertible which implies $\{P_{n_{\circ}}e_r:1\leq r\leq k\}$ is linearly independent. It contradicts the fact that $\text{dim ran}P_n<k$ for all $n$.
\end{proof}

\begin{remark}
Let $\mathscr{H}$ be an infinite dimensional Hilbert space with orthonormal basis $\{e_r:r\geq 1\}$. Suppose $P_n$ be projection on $\text{span}\{e_1,e_1,\cdots,e_n\}$. Then $P_n\xrightarrow{\text{SOT}}I$ where $I$ is the identity operator on $\mathscr{H}$. Note that $\text{dim ran}P_n$ is finite for all $n$ but $\text{ran}E(I)=\mathscr{H}$ is infinite dimensional. It shows that the above lemma does not hold for $k=\infty$.
\end{remark}

There are various versions of the spectral theorem. We will use the following one.

\begin{theorem}(Spectral Theorem)
Let $T\in\mathscr{B\mathscr{(H)}}$ be normal. Then there exists a unique spectral measure $E$ defined on the Borel subsets of $\mathbb{C}$ and supported on $\sigma(T)$ such that $$T=\int_{\mathbb{C}}x\;dE(x)$$  
Furthermore, every projection $E(\omega)$ commutes with every $S\in\mathscr{B{\mathscr(H)}}$ which commutes with $T$ where $\omega$ is a Borel set in $\mathbb{C}$.
\end{theorem}

A good reference for the spectral measure (resolution of identity) and the spectral theorem is \cite{WR}. Finally, we will end this section by recalling the following proposition.

\begin{proposition}[Proposition $2.3$, \cite{RM}]\label{b}
Let $T\in\mathscr{B\mathscr{(H)}}$ and $k\in\mathbb{N}_{\infty}$. Then
\begin{align*}
\Lambda_k(T)\subseteq\bigcap\limits_{V\in\mathscr{V}_k} W(V^*TV)
\end{align*}
where $\mathscr{V}_k$ is the set of all isometries $V:\mathscr{H}\rightarrow\mathscr{H}$ such that codimension of ran $V$ is less than $k$.
\end{proposition}

\section{Higher rank numerical range of a normal operator}

Let us start with the following definition.

\begin{definition}
A subset $H\subseteq\mathbb{C}$ is said to be a half closed-half plane in $\mathbb{C}$ if $H$ can be written as the union of an open half plane $\omega$ and a ray $l$ on the line $\delta\omega$. Moreover, if $\lambda$ be the initial point of $l$ then we sometimes call $H$ as a half closed-half plane at $\lambda$ generated by the line $\delta \omega$. 
\end{definition}

For example, $\{z\in\mathbb{C}:\Im{(z)}>0\}\cup[0,\infty)$ is a half closed-half plane at $0$ generated by the real axis.

Let $T\in\mathscr{B(\mathscr{H})}$ be a normal operator and $E$ be the unique spectral measure associated with $T$ defined on the Borel $\sigma$-algebra in $\mathbb{C}$ supported on $\sigma(T)$.  Suppose
\begin{align}\label{defining S_k}
\mathcal{S}_k:=\left\{H: H\text{ is a half closed-half plane in }\mathbb{C} \text{ and dim ran }E(H)<k \right\}.
\end{align}
Define
\begin{align}
\Omega_k(T):=\bigcap\limits_{H\in\mathcal{S}_k} H^c.
\end{align}
As arbitrary intersection of convex set is always convex, $\Omega_k(T)$ is a convex set in $\mathbb{C}$. It can be checked that $\Omega_k(T)$ is non-empty for $k\in\mathbb{N}$ (cf. Proposition \ref{nonemptyness}). Later, we will also see an example that $\Omega_{\infty}(T)$ could be empty (Example \ref{infinity rank numerical range being empty}).

We are now ready to state the main theorem of this paper. A finite dimensional version of our result is found at \cite{LS} (see also \cite{CKZ2},\cite{GLP}). It also extends the description of the higher rank numerical range of self-adjoint operator acting on an infinite dimensional Hilbert space \cite{RM}.

\begin{theorem}\label{main theorem}
Let $T\in\mathscr{B(\mathscr{H})}$ be a normal operator and $k\in\mathbb{N}_{\infty}$. Then 
\begin{align*}
\Lambda_k(T)=\bigcap\limits_{V\in\mathscr{V}_k}W(V^*TV)=\Omega_k(T)
\end{align*}
where $\mathscr{V}_k$ is the set of all isometries $V:\mathscr{H}\rightarrow\mathscr{H}$ such that codimension of ran $V$ is less than $k$.
\end{theorem}

The proof of Theorem \ref{main theorem} will be given in the next section. It has many applications. In this section, we, motivated by the description of the numerical range of normal operators by Durszt \cite{D}, derive an analogous description for the higher rank numerical range of normal operators as an immediate consequence of Theorem \ref{main theorem}.

\begin{corollary}\label{analogous description}
Let $T\in\mathscr{B}\mathscr{(H)}$ be normal and $k\in\mathbb{N}_{\infty}$. Suppose $\mathcal{R}_k$ is the collection of all convex Borel set $s$ with $\text{dim ran}E(s^c)<k$. Then $\Lambda_k(T)=\bigcap\limits_{s\in\mathcal{R}_k}s$.
\end{corollary}

\begin{proof}
Let $V_k(T)=\bigcap\limits_{s\in\mathcal{R}_k}s$. Let $H$ be a half closed-half plane with $\text{dim ran}E(H)<k$. As $H^c$ is a convex Borel set with $\text{dim ran}E(H)<k$, we have $H^c\in\mathcal{R}_k$. So, $V_k(T)\subseteq\Omega_k(T)$. Hence by Theorem \ref{main theorem}, we obtain $V_k(T)\subseteq\Lambda_k(T)$.

Let $\lambda\notin V_k(T)$. Then there exists a convex Borel set $s$ with $\text{dim ran}E(s^c)<k$ such that $\lambda\notin s$. As $s$ is convex and $\lambda\notin s$, there exists a closed half plane $H$ with the line $\delta H$ passing through $\lambda$ such that $s\subseteq H$ by Lemma \ref{separating a convex set}. Let $l_1,l_2$ be two opposite directed rays with initial point $\lambda$ on the line $\delta H$. Now, $s$ can not have non empty intersection with both $l_1$ and $l_2$ as otherwise $\lambda\in s$ because of convexity of $s$. Without loss of generality, let $s\cap l_1=\emptyset$. Take $K=H^c\cup l_1\subseteq s^c$. As $\text{dim ran}E(s^c)<k$, we have $\text{dim ran}E(K)<k$. So, $K$ is a half closed-half plane at $\lambda$ with $\text{dim ran}E(K)<k$. Now, by Theorem \ref{main theorem} and the definition of $\Omega_k(T)$, we have $\Lambda_k(T)=\Omega_k(T)\subseteq K^c$. As $\lambda\notin K^c$, we have $\lambda\notin\Lambda_k(T)$. This completes the proof.
\end{proof}

\section{Proof of the Main Theorem}

We begin with few lemmas.

\begin{lemma}\label{equivalent criterion of being member of our set}
Let $T\in\mathscr{B(\mathscr{H})}$ be a normal operator and $k\in\mathbb{N}_{\infty}$. Then $\lambda\in\Omega_k(T)$ if and only if for any half closed-half plane $H$ at $\lambda$, $\text{dim ran }E(H)\geq k$.
\end{lemma}
\begin{proof}
`$\Rightarrow$'. Let there exists a half closed-half plane $H$ at $\lambda$ such that $\text{dim ran }E(H)<k$. Then by the definition of $\Omega_k(T)$, we have $\Omega_k(T)\subseteq H^c$. As $\lambda\notin H^c$, we have $\lambda\notin\Omega_k(T)$.\\
`$\Leftarrow$'. Let $\lambda\notin\Omega_k(T)$. Then by the definition of $\Omega_k(T)$, there exists a half closed-half plane $H$ containing $\lambda$ such that $\text{dim ran}E(H)<k$. Let $L$ be the line passing through $\lambda$ and parallel to the line $\delta H$. Consider a half closed-half plane $H'$ at $\lambda$  generated by $L$ such that $H'\subseteq H$. Then $\text{dim ran}E(H')<k$.
\end{proof}

Let $T$ be normal and $A,B\subseteq\mathbb{C}$. We denote $A\subseteq B$ a.e. to mean $E(A\setminus B)=0$.

\begin{lemma}\label{boundary points}
Let $T\in\mathscr{B(\mathscr{H})}$ be a normal operator and $k\in\mathbb{N}$. Suppose $\lambda\in\Omega_k(T)$. Then $\lambda\in\delta\Omega_k(T)$ if and only if there exists a half closed-half plane $H$ at $\lambda$ such that $\text{dim ran}E(H^\circ)<k$.
\end{lemma}

\begin{proof}
`$\Leftarrow$'. Let $\lambda\notin\delta\Omega_k(T)$. Then $\lambda\in\Omega_k(T)^\circ$ as $\Omega_k(T)=\Omega_k(T)^\circ\cup\left(\delta\Omega_k(T)\cap\Omega_k(T)\right)$ by Lemma \ref{partition of disjoint set}. Let $H$ be a half closed-half plane at $\lambda$. As $\lambda\in\Omega_k(T)^\circ$, it is possible to choose a point $\mu\in\Omega_k(T)$ and a half closed-half plane $H'$ at $\mu$ such that $H'\subseteq H^\circ$. Now, by Lemma \ref{equivalent criterion of being member of our set}, we have $\text{dim ran}E(H')\geq k$ and hence $\text{dim ran}E(H^\circ)\geq k$ (note that we do not need $k$ to be finite to prove this direction).\\
`$\Rightarrow$'. Let $\lambda\in\delta\Omega_k(T)$. Then there exists a sequence $\{\lambda_n\}_{n=1}^\infty$ not in $\Omega_k(T)$ and lying on a ray with initial point $\lambda$ (which is possible as $\Omega_k(T)$ is convex) such that $\lim\limits_{n\to\infty}\lambda_n=\lambda$. Since $\lambda_n\notin\Omega_k(T)$, by Lemma \ref{equivalent criterion of being member of our set}, there exists a half closed-half plane $H_n$ at $\lambda_n$ such that $\text{dim ran}E(H_n)<k$ for each $n$. Let $L_n=\delta H_n$. Suppose $L_n'$ be the line passing through $\lambda$ and parallel to $L_n$ and $H_n'$ be a half closed-half plane generated by $L_n'$ lying on the side of $H_n$. Let $\theta_n$ be the slope of $L_n'$. As $\theta_n\in [0,\pi]$ for all $n$ and $[0,\pi]$ is compact, $\{\theta_n\}_{n=1}^\infty$ has a convergent subsequence which we call $\{\theta_n\}_{n=1}^\infty$ itself and let $\lim\limits_{n\to\infty}\theta_n=\theta$. Let $L$ be the line with slope $\theta$ and passing through $\lambda$. Suppose $H$ is a half closed-half plane generated by $L$ such that $H^\circ$ lies on the side of $H_n$ for all but finitely many $n$. Now, we will show $H$ is the desired half closed-half plane at $\lambda$.

Let $\text{dim ran}E(H^\circ)\geq k$. We claim that there exists a closed half plane $\omega$ generated by a line $L'$ parallel to $L$ such that $\omega\subseteq H^\circ$ and $\text{dim ran}E(\omega)\geq k$. Indeed, if it be not true then $H^\circ=\bigcup\limits_{n=1}^\infty\omega_n$ where $\{\omega_n\}_{n=1}^\infty$ is an increasing sequence of closed half plane with $\text{dim ran}E(\omega_n)<k$ for all $n$. Then $E(\omega_n)\xrightarrow{\text{SOT}}E(H^\circ)$. As $\text{dim ran}E(\omega_n)<k$ for all $n$, by Lemma \ref{SOT limit of projections}, we obtain $\text{dim ran}E(H^\circ)<k$  which is a contradiction. Hence the claim follows.

Let $Q$ be a closed and bounded rectangle such that $\sigma(T)\subseteq Q$. So, $E(Q^c)=0$. Let $S$ be a strip generated by $L'$ and the mirror image of $L'$ with respect to $L$. As $\lim\limits_{n\to\infty}\theta_n=\theta$ and $L_n$ being parallel to $L_n', \lim\limits_{n\to\infty}\lambda_n=\lambda, E(Q^c)=0$, there exists $n_{\circ}\in\mathbb{N}$ such that $H_{n_{\circ}}'\setminus H_{n_\circ}\subseteq S$ a.e. which implies $\omega\subseteq H_{n_{\circ}}$ a.e. As $\text{dim ran}E(H_{n_{\circ}})<k$, we have $\text{dim ran}E(\omega)<k$ which is a contradiction. Hence $\text{dim ran}E(H^\circ)<k$.
\end{proof}

\begin{lemma}\label{relation between k and infinity}
Let $T\in\mathscr{B(\mathscr{H})}$ be a normal operator. Then $\Omega_\infty(T)=\bigcap\limits_{k\geq 1}\Omega_k(T)$.
\end{lemma}
\begin{proof}
It is clear that $\Omega_{\infty}(T)\subseteq\bigcap\limits_{k\geq 1}\Omega_k(T)$ by Lemma \ref{equivalent criterion of being member of our set}. Let $\lambda\in\bigcap\limits_{k\geq 1}\Omega_k(T)$. Then $\lambda\in\Omega_k(T)$ for all $k\geq 1$. Let $H$ be a half closed-half plane at $\lambda$. Then by Lemma \ref{equivalent criterion of being member of our set}, $\text{dim ran }E(H)\geq k$ for all $k\geq 1$. So, $\text{dim ran }E(H)\geq \infty$. Again, by Lemma \ref{equivalent criterion of being member of our set}, we get $\lambda\in\Omega_{\infty}(T)$.
\end{proof}

 Let us now observe the following property before going to the proof of the main theorem.

\begin{proposition}\label{nonemptyness}
Let $T\in\mathscr{B(\mathscr{H})}$ be normal and $k\in\mathbb{N}$. Then $\Omega_k(T)$ is non-empty.
\end{proposition}

\begin{proof}
Let $V_k(T)=\bigcap\limits_{H\in\mathcal{S}_k}\overline{H^c}$. Then note that $\lambda\notin V_k(T)$ if and only if there exists a half closed half plane $H$ with $\text{dim ran}E(H)<k$ such that $\lambda\in H^\circ$. Clearly $\Omega_k(T)\subseteq V_k(T)$. Let us first show that $V_k(T)\neq\emptyset$. Suppose $V_k(T)=\emptyset$. Then $\mathbb{C}=\bigcup\limits_{H\in\mathcal{S}_k} H^{\circ}$. Let $Q$ be a closed and bounded rectangle in $\mathbb{C}$ such that $\sigma(T)\subseteq Q$. Then $\text{dim ran}E(Q^c)=0$. Now, $Q=\bigcup\limits_{H\in\mathcal{S}_k} (H^{\circ}\cap Q)$. As $Q$ is compact, there exist $H_1,\cdots,H_n\in\mathcal{S}_k$ such that $Q=\bigcup\limits_{i=1}^n({H_i}^{\circ}\cap Q)$. Since ${H_i}^{\circ}\cap Q\subseteq H_i$ and $H_i\in\mathcal{S}_k$ for all $i\in\{1,2,\cdots,n\}$, we have $\text{dim ran}E(Q)$ is finite which is a contradiction. So, $V_k(T)$ is a non-empty, closed and convex set. We will now consider the following cases and show that $\Omega_k(T)\neq\emptyset$ in each case .

\textbf{Case I:} Let $V_k(T)$ be a singleton set. Without loss of generality, we may take $V_k(T)=\{0\}$. It is sufficient to show $\text{dim ran}E(\mathbb{C}\setminus\{0\})$ is finite as then $0$ is an eigenvalue of $T$ with infinite multiplicity and hence $0\in\Omega_k(T)$. Let $(p_n)_{n=1}^\infty$ be a sequence of positive real numbers converging to $0$. As $p_n\notin V_k(T)$, there exists a half closed-half plane $H_n$ at $p_n$ such that $\text{dim ran}E(H_n)<k$. Let $L_n$ be the line passing through $0$ and parallel to the line $\delta H_n$.  Suppose $\theta_n$ be the slope of $L_n$. As $\theta_n\in [0,\pi]$ for all $n$ and $[0,\pi]$ is compact,  $(\theta_n)_{n=1}^\infty$ has a convergent subsequence say $(\theta_n)_{n=1}^\infty$ itself and $\lim\limits_{n\to\infty}\theta_n=\theta$. As $k\in\mathbb{N}$, by the same reason argued in Lemma \ref{boundary points}, there exists a half closed-half plane $H$ at $0$ with $\theta$ being the slope of the line $\delta H$ such that $\text{dim ran}E(H^\circ)<k$. Clearly $H$ contains $p_n$ for all $n$.

Now, by applying the same argument to a sequence of negative real numbers converging to $0$, there exists a half closed-half plane $K_1$ at $0$ such that $\text{dim ran}E(K_1^\circ)<k$. If $\delta K_1=\delta H$ then $\text{dim ran}E(\mathbb{C}\setminus\{0\})<4k$ as $\text{dim ran}E((\delta K_1)^c)<2k$ and by applying the same argument to a sequence converging to $0$ lying on each $L_i\setminus\{0\}$, we have $\text{dim ran}E(L_i\setminus\{0\})<k$ for all $i\in\{1,2\}$  where $L_1, L_2$ be are two opposite directed rays with initial point $0$ on the line $\delta K_1$. If $\delta K_1\neq\delta H$ then there exists a triangle $\Delta_1$ with $0$ as one of the vertices such that $\text{dim ran}E(\Delta_1^c)<2k$. Now, choose a sequence $(q_n)_{n=1}^\infty$ converging to $0$ lying on $l\setminus\{0\}$ where $l$ is the line segment joining $0$ and the midpoint of the side of $\Delta_1$ opposite to $0$. Again applying the same argument to $(q_n)_{n=1}^\infty$, there exists a half closed-half plane $K_2$ at $0$ such that $\text{dim ran}E(K_2^\circ)<k$. If the line $\delta K_2$ does not intersect $\Delta_1$ other than $0$ then we get $\text{dim ran}E(\mathbb{C}\setminus \{0\})<3k$ as $\text{dim ran}E(\Delta_1^c)<2k$ and $\text{dim ran}E(\Delta_1\setminus\{0\})\leq \text{dim ran}E(K_2^\circ)<k$ (since $\Delta_1\setminus\{0\}\subseteq K_2^\circ$). Now, if the line $\delta K_2$ passes through any side of $\Delta_1$ adjacent to $0$, then we have $\text{dim ran}E(\mathbb{C}\setminus \{0\})<4k$ as $\text{dim ran}E((\delta K_2)^c)<2k$ and $\text{dim ran}E(M_i\setminus\{0\})<k$ for all $i\in\{1,2\}$  where $M_1, M_2$ be are two opposite directed rays with initial point $0$ on the line $\delta K_2$. Finally, if the line $\delta K_2$ intersects $\Delta_1^\circ$ then there exists another triangle $\Delta_2$ with $0$ as one of the vertices such that $\Delta_2\subsetneq\Delta_1$ and $\text{dim ran}E(\Delta_2^c)<2k$. If the final case continues to occur repeatedly then we will end up with $\text{dim ran}E(\mathbb{C}\setminus \{0\})<4k$. Therefore, in any case, we have $\text{dim ran}E(\mathbb{C}\setminus\{0\})<4k$. Hence we are done in this case.

\textbf{Case II:} Let $V_k(T)$ be a nondegenerate line segment. Without loss of generality, we may assume $V_k(T)=\left[-p,p\right]\subseteq\mathbb{R}$ where $p>0$. Let $\omega_1:=\{z\in\mathbb{C}:\Im(z)>0\}$. Let $(z_n)_{n=1}^\infty$ be a sequence lying on the positive imaginary axis converging to $0$. Claim: $\text{dim ran}E(\omega_1)<k$. If not then as $k\in\mathbb{N}$, there exists a closed half plane $\omega$ generated by a line parallel to the real axis such that $\omega\subseteq\omega_1$ and $\text{dim ran}E(\omega)\geq k$ (by the similar reason argued in the proof of Lemma \ref{boundary points}). As $z_n\notin V_k(T)$, there exists a half closed-half plane $H_n$ with $\text{dim ran}E(H_n)<k$ such that $z_n\in H_n^{\circ}$. Observe that the line $\delta H_n$ can not intersect $(-p,p)$. Indeed, if so at $q\in(-p,p)$ then either $\frac{q+p}{2}\notin V_k(T)$ or $\frac{-p+q}{2}\notin V_k(T)$ as $V_k(T)\subseteq \overline{H_n^c}$. This is a contradiction. Now, since $\delta H_n$ does not intersect $(-p,p)$ for every $n\in\mathbb{N}, \lim\limits_{n\to\infty}z_n=0$ and $E(Q^c)=0$, it follows that there exists $n_{\circ}\in\mathbb{N}$ such that $\omega\subseteq H_{n_{\circ}}$ a.e. which implies $\text{dim ran}E(\omega)<k$. This is a contradiction. So, $\text{dim ran}E(\omega_1)<k$. Similarly, we have $\text{dim ran}E(\omega_2)<k$ where $\omega_2:=\{z\in\mathbb{C}:\Im(z)<0\}, \text{ dim ran}E(p,\infty)<k$ and $\text{ dim ran}E(-\infty,-p)<k$. Therefore,  $\text{dim ran}E([-p,p]^c)<4k$ which implies $\text{dim ran}E([-p,p])$ is infinite. Now, as $k\in\mathbb{N}$, by Lemma 3.2, \cite{RM} and Lemma \ref{equivalent criterion of being member of our set}, we have $\emptyset\neq\Omega_k(P_{\text{ran}E([-p,p])}T|_{\text{ran}E([-p,p])})\subseteq\Omega_k(T)$. Hence $\Omega_k(T)\neq\emptyset$ for $k\in\mathbb{N}$.

\textbf{Case III:} Let $V_k(T)$ contains atleast three non-collinear points say $p,q$ and $r$. As $V_k(T)$ is a closed and convex set, $V_k(T)$ contains the closed triangle $\Delta$ with vertices $p,q$ and $r$. Let $s$ be an interior point of $\Delta$. If possible let $s\notin\Omega_k(T)$. By Lemma \ref{equivalent criterion of being member of our set}, there exist a half closed-half plane $H$ at $s$ such that $\text{dim ran}E(H)<k$. Then $V_k(T)\subseteq\overline{H^c}$, which contradicts that $V_k(T)$ contains $\Delta$. So, $s\in\Omega_k(T)$. Hence we are done in this case also.
\end{proof}

\begin{remark}
Let $k\in\mathbb{N}$. Observe that if $V_k(T)=\{0\}$ then $0$ is an eigenvalue of $T$ with infinite multiplicity. However, this is not true for $k=\infty$ as discussed later (Remark \ref{r}).
\end{remark}

Let $\lambda_1,\lambda_2\in\mathbb{C}$. We denote by $(\lambda_1,\lambda_2)=\{t\lambda_1+(1-t)\lambda_2:0<t<1\}$, the open line segment joining $\lambda_1$ and $\lambda_2$ and $[\lambda_1,\lambda_2]=\{t\lambda_1+(1-t)\lambda_2:0\leq t\leq 1\}$, the closed line segment joining $\lambda_1$ and $\lambda_2$. We are now ready to prove the main theorem of this paper.

\begin{proof}[\textit{\textbf{Proof of Theorem \ref{main theorem}}}]

We will first prove 
\begin{align*}
\Lambda_k(T)\subseteq\bigcap\limits_{V\in\mathscr{V}_k}W(V^*TV)\subseteq\Omega_k(T)
\end{align*}
for $k\in\mathbb{N}_\infty$. Let $k\in\mathbb{N}_\infty$ and $H$ be a half closed-half plane with $\text{dim ran}E(H)<k$.  Since codimension of $\text{ran}E(H^c)$ is finite and $\mathscr{H}$ is infinite dimensional, there exists an isometry $V_H$ with range $\text{ran}E(H^c)$. Clearly $V_H\in\mathscr{V}_k$ and $V_HV_H^*=E(H^c)$. Denote $\mathscr{K}=\text{ran}E(H^c)$. Observe that $W(V_H^*TV_H)=W(P_{\mathscr{K}}T|_{\mathscr{K}})$. Let $\lambda\in W(P_{\mathscr{K}}T|_{\mathscr{K}})$. Then there exists $f\in\mathscr{K}$ with $\Vert f\Vert=1$ such that
\begin{align*}
\lambda=\langle P_{\mathscr{K}}T|_{\mathscr{K}}f,f\rangle=\langle Tf,f\rangle=\int_{\mathbb{C}}x\;d\langle E(\centerdot)f, f\rangle=\int_{H^c}x\;d\Vert E(\centerdot)f\Vert^2
\end{align*}
Also $E_{f,f}(H^c)=\langle E(H^c)f,f\rangle=\Vert f\Vert^2=1$. So, by Lemma \ref{positive measure}, we have $\lambda\in H^c$. Hence we obtain $W( V_H^*TV_H)\subseteq H^c$. Now, using Proposition \ref{b}, we have
\begin{align*}
\Lambda_k(T)\subseteq\bigcap\limits_{V\in\mathscr{V}_k} W(V^*TV)\subseteq\bigcap\limits_{H\in\mathcal{S}_k}W( V_H^*TV_H)\subseteq \bigcap\limits_{H\in\mathcal{S}_k} H^c=\Omega_k(T).
\end{align*}

We will now prove
\begin{align}\label{e}
\Omega_k(T)\subseteq\Lambda_k(T)
\end{align}
for $k\in\mathbb{N}$. We will show it by induction on $k$. Let $\lambda\in\Omega_1(T)$.  If possible let $\lambda\notin W(T)$. As $W(T)$ is convex, by Lemma \ref{separating a convex set}, there exists a closed half plane $H'$ with the line $\delta H'$ passing through $\lambda$ such that $W(T)\subseteq H'$. Let $L_1', L_2'$ be two opposite directed rays with initial point $\lambda$ on the line $\delta H'$. Now, $L_1'$ and $L_2'$ both can not have non empty intersection with $W(T)$ as otherwise $\lambda\in W(T)$ because of convexity of $W(T)$. So, without loss of generality, let $L_1'$ has empty intersection with $W(T)$. Let $\mathscr{K}_1=\text{ran}E(L_1')$. We claim that $\text{dim}\mathscr{K}_1=0$. If $\text{dim }\mathscr{K}_1$ be positive and finite then $P_{\mathscr{K}_1}T|_{\mathscr{K}_1
}$ has eigenvalue on $L_1'$ which implies $T$ has eigenvalue on $L_1'$ (since $\mathscr{K}_1$ is reducible subspace with respect to $T$) which contradicts that $L_1'$ has empty intersection with $W(T)$. If $\text{dim }\mathscr{K}_1$ be infinite then $\emptyset\neq W\left(P_{\mathscr{K}_1}T|_{\mathscr{K}_1}\right)\subseteq W(T)\cap L_1'$ which again contradicts that $L_1'$ has empty intersection with $W(T)$. So, $\text{dim }\mathscr{K}_1=0$. Then $\widetilde{H}={H'}^c\cup L_1'$ is a half closed-half plane at $\lambda$ with $\text{dim ran}E(\widetilde{H})=0$. So, by Lemma \ref{equivalent criterion of being member of our set}, $\lambda\notin\Omega_1(T)$ which contradicts our hypothesis. Hence $\lambda\in W(T)$. It shows (\ref{e}) is true for $k=1$.

Let us now assume that (\ref{e}) is true for any $m<k$. Let us first show $\Omega_k(T)\subseteq\overline{\Lambda_k(T)}$. Since for any $a,b\in\mathbb{C},\; \overline{\Lambda_k(aT+bI)}=a\overline{\Lambda_k(T)}+b,$ $\Omega_k(aT+bI)=a\Omega_k(T)+b$ and $aT+bI$ is normal, without loss of generality we show that whenever $0\in\Omega_k(T),$ it would imply $0\in\overline{\Lambda_k(T)}$. Let $0\in\Omega_k(T)$. Let $\theta\in[0,2\pi)$. Suppose $L_{\theta}$ be a straight line passing through $0$ making angle $\theta$ with positive imaginary axis. Let $\omega_1^\theta,\omega_2^\theta$ are two open half planes induced by $L_{\theta}$ and $L_1^\theta,L_2^\theta$ are two opposite directed rays with initial point $0$ on $L_\theta$. Then $H_1^\theta=\omega_1^\theta\cup L_1^\theta$ and $H_2^\theta=\omega_2^\theta\cup L_2^\theta$ are two half closed-half planes at $0$. As $0\in\Omega_k(T)$, we have $\text{dim ran}E(H_j^\theta)\geq k$ for all $j\in\{1,2\}$ by Lemma \ref{equivalent criterion of being member of our set}. Let $S_\theta=\Re{(e^{i\theta}T)}$ and $E_S^\theta$ be the spectral measure associated with $S_\theta$. Then we have $\text{dim ran}E_S^\theta\left(\Re(e^{i\theta}H_j^\theta)\right)\geq k$ for all $j\in\{1,2\}$ where we denote $\Re(C):=\{\Re(z):z\in C\}$ for some $C\subseteq\mathbb{C}$. Let 
\begin{align*}
& A_k^\theta:=\{a_{\theta}\in\mathbb{R}:\text{dim ran}E_S^\theta(-\infty,a_{\theta}]<k\},\\
& B_k^\theta:=\left\{b_{\theta}\in\mathbb{R}:\text{dim ran}E_S^\theta[b_{\theta},\infty)<k\right\}.
\end{align*}
As $\text{dim ran}E_S^\theta\left(\Re(e^{i\theta}H_j^\theta)\right)\geq k$ for all $j\in\{1,2\}$, we obtain $0\notin A_k^\theta$ and $0\notin B_k^\theta$. So, $0\in {A_k^\theta}^c\cap {B_k^\theta}^c=\Lambda_k(\Re({e^{i\theta}T)})$ by Theorem \ref{result for selfadjoint operator}. Then there exists an isometry $V_{\theta}:\mathbb{C}^k\rightarrow\mathscr{H}$ such that ${V_{\theta}}^*\Re({e^{i\theta}T)}V_{\theta}=0I_k$. So, $0=\lambda_k({V_{\theta}}^*\Re({e^{i\theta}T)}V_{\theta})\leq\lambda_k(\Re({e^{i\theta}T)})$ for all $\theta\in[0,2\pi)$. By Theorem \ref{geomertic description for operator}, we get $0\in\overline{\Lambda_k(T)}$. Hence we obtain $\Omega_k(T)\subseteq\overline{\Lambda_k(T)}$. As $\Lambda_k(T)$ is a convex set in $\mathbb{C}$, we have $\overline{\Lambda_k(T)}^\circ=\Lambda_k(T)^\circ$ by Lemma \ref{closure interior of a convex set}. So,
\begin{align}\label{part a}
\Omega_k(T)^\circ\subseteq\Lambda_k(T)^\circ
\end{align}

Next, we will show $\delta\Omega_k(T)\cap\Omega_k(T)\subseteq\Lambda_k(T)$. Let $\lambda\in\delta\Omega_k(T)\cap\Omega_k(T)$.  By Lemma \ref{boundary points}, there exist a half closed-half plane $H$ at $\lambda$ with $\text{dim ran }E(H^\circ)<k$. Let $L_1,L_2$ be two opposite directed rays with initial point $\lambda$ on the line $\delta H$. Then $H_1=H^\circ\cup L_1 $ and $H_2=H^\circ\cup L_2$ are two half closed-half planes at $\lambda$. As $\lambda\in\Omega_k(T)$, we have $\text{dim ran}E(H_1)\geq k$ and $\text{dim ran}E(H_2)\geq k$ by Lemma \ref{equivalent criterion of being member of our set}. As $\text{dim ran}E(H^\circ)<k$, we have $\text{dim ran}E(L_1)\geq 1$ and $\text{dim ran}E(L_2)\geq 1$. Let us now consider the following cases.

\textbf{Case I:} Let $\lambda$ be an eigenvalue of $T$ with eigenvector $f$. Let $T=[\lambda]\oplus T_0$ with respect to the decomposing $\mathscr{H}=\text{span}\{f\}\oplus\text{span}\{f\}^\bot$. Clearly $T_0$ is normal. As $\lambda\in\Omega_k(T)$, we have $\lambda\in\Omega_{k-1}(T_0)$ by Lemma \ref{equivalent criterion of being member of our set}. So, by induction hypothesis, we have $\lambda\in\Lambda_{k-1}(T_0)$. Therefore, there exists a projection $P_0$ of rank $k-1$ such that $P_0T_0P_0=\lambda P_0$. Take $P=[1]\oplus P_0$. Then $P$ is a projection of rank $k$ such that $PTP=\lambda P$. Hence $\lambda\in\Lambda_k(T)$.

\textbf{Case II:} Let $\lambda$ be not an eigenvalue of $T$ but $T$ has eigenvalues both on $L_1$ and $L_2$ say $\lambda_1$ and $\lambda_2$ respectively. Let $f_1,f_2$ be two orthonormal eigenvectors of $T$ corresponding to eigenvalues on $\lambda_1, \lambda_2$ respectively. Let $T=T_1\oplus T_2$ with respect to the decomposition of $\mathscr{H}=\text{span}\{f_1,f_2\}\oplus\text{span}\{f_1,f_2\}^\bot$. Clearly, $T_1, T_2$ are both normal and $T_1$ is unitarily similar to $\text{diag}(\lambda_1,\lambda_2)$. As $\lambda\in\Omega_k(T)$, we have $\lambda\in\Omega_1(T_1)$ and $\lambda\in\Omega_{k-1}(T_2)$ by Lemma \ref{equivalent criterion of being member of our set}. Then by induction hypothesis, we have  $\lambda\in\Lambda_1(T_1)$ and $\lambda\in\Lambda_{k-1}(T_2)$. So, there exists an unit vector $\{g_1\}$ in $\text{span}\{f_1,f_2\}$ such that $\langle T_1g_1,g_1\rangle=\lambda$ and an orthonormal set $\{g_j\}_{j=2}^k$ in $\text{ span}\{f_1,f_2\}^\bot$ such that $\langle T_2g_j,g_r\rangle=\lambda\delta_{j,r}$ for $j,r\in\{2,\cdots,k\}$. Then $\{g_j\}_{j=1}^k$ is an orthonormal set and $\langle Tg_j,g_r\rangle=\lambda\delta_{j,r}$ for all $j,r\in\{1,2,\cdots,k\}$. Hence $\lambda\in\Lambda_k(T)$.

\textbf{Case III:} Suppose $T$ does not have eigenvalues on $L_1$ and $L_2$ both. If both $\text{dim ran}E(L_1)$ and $\text{dim ran}E(L_2)$ are atleast $k,$ then by Theorem \ref{result for selfadjoint operator}, we have 
\begin{align*}
\lambda\in\Lambda_k\left(P_{\text{ran}E(L)}T|_{\text{ran }E(L)}\right)\subseteq\Lambda_k(T).
\end{align*}

Let one of $\text{dim ran}E(L_1)$ and $\text{dim ran}E(L_2)$ be less than $k$. Without loss of generality, let $\text{dim ran}E(L_2)=r<k$. Then $E(L_2)TE(L_2)$ has $r$ eigenvalues on $L_2$ say $\lambda_1,\lambda_2,\cdots,\lambda_r$ with $0<\vert\lambda-\lambda_1\vert\leq\cdots\leq\vert\lambda-\lambda_r\vert$. As $T$ does not have eigenvalues on $L_1$ and $L_2$ both, $\text{ dim ran}E(L_1)$ is infinite. Then there exists $\mu\in L_1$ such that $\text{dim ran}E([\lambda,\mu])$ and $\text{dim ran}E\left([\lambda,\mu]^c\cap L_1\right)$ are both infinite. Let $T=T_1\oplus T_2$ with respect to the decomposing $\mathscr{H}=\text{ran}E([\lambda_1,\mu])\oplus\text{ran}E([\lambda_1,\mu])^\bot$. Clearly $T_1,T_2$ are both normal. As $\lambda\in\Omega_k(T)$, we have $\lambda\in\Omega_1(T_1)$ and $\lambda\in\Omega_{k-1}(T_2)$ by Lemma \ref{equivalent criterion of being member of our set}. So, by induction hypothesis, we have  $\lambda\in\Lambda_1(T_1)$ and $\lambda\in\Lambda_{k-1}(T_2)$. Then there exists an unit vector $\{g_1\}$ in $\text{ran}E([\lambda_1,\mu])$ such that $\langle T_1g_1,g_1\rangle=\lambda$ and an orthonormal set $\{g_j\}_{j=2}^k$ in $\text{ran}E([\lambda_1,\mu])^\bot$ such that $\langle T_2g_j,g_r\rangle=\lambda\delta_{j,r}$ for $j,r\in\{2,\cdots,k\}$. Then $\{g_j\}_{j=1}^k$ is an orthonormal set and $\langle Tg_j,g_r\rangle=\lambda\delta_{j,r}$ for all $j,r\in\{1,2,\cdots,k\}$. Hence $\lambda\in\Lambda_k(T)$.

So, combining all three cases, we conclude
\begin{align}\label{part b}
\delta\Omega_k(T)\cap\Omega_k(T)\subseteq\Lambda_k(T)
\end{align}
Now, by (\ref{part a}), (\ref{part b}) and Lemma \ref{partition of disjoint set}, we obtain $\Omega_k(T)\subseteq\Lambda_k(T)$. Hence by principle of mathematical induction, we have $\Omega_k(T)\subseteq\Lambda_k(T)$ for $k\in\mathbb{N}$.

Finally, we will prove $\Omega_{\infty}(T)\subseteq\Lambda_{\infty}(T)$. Let $\lambda\in\Omega_{\infty}(T)$. Then by Lemma \ref{relation between k and infinity}, $\lambda\in\Omega_k(T)$ for all $k\geq 1$. By (\ref{e}), we have $\lambda\in\Lambda_k(T)$ for all $k\geq 1$. So, $\lambda\in\bigcap\limits_{k\geq 1}\Lambda_k(T)=\Lambda_{\infty}(T)$.
\end{proof}

\section{Examples and Consequences}

We compute here the higher rank numerical ranges of some standard normal operators using Theorem \ref{main theorem}. Our first example is the bilateral shift.

\begin{example}\label{bilateral shift}
Let $S:l^2(\mathbb{Z})\rightarrow l^2(\mathbb{Z})$ such that 
\begin{align*}
S\left((\ldots,x_{-1},\boxed{x_0},x_{1},\ldots)\right)&=(\ldots,x_{-2},\boxed{x_{-1}},x_{0},\ldots),\\
&\qquad\text{ where } (\ldots,x_{-1},\boxed{x_0},x_{1},\ldots)\in l^2(\mathbb{Z}).
\end{align*}
Note that $S$ is unitary with $\sigma_e(S)=\emptyset$ and $\sigma(S)=\mathbb{T}$. Let $S_n$ be $n$-dimensional backward shift. As $S_n$ is a compression of $S$, we have $W(S_n)=\{z\in\mathbb{C}:\vert z\vert\leq\cos\frac{\pi}{n+1}\}\subseteq W(S)$. Taking $n\to\infty$, we obtain $\mathbb{D}\subseteq W(S)$. Again, $W(S)\subseteq\overline{W(S)}=\text{conv }\sigma(S)=\overline{\mathbb{D}}$. Let $\lambda\in\mathbb{T}$ such that $\lambda\in W(S)$. Then there exists $x\in l^2(\mathbb{Z})$ with $\Vert x\Vert=1$ such that $\lambda=\langle Sx,x\rangle$. Now,
\begin{align*}
1=\vert\lambda\vert=\vert\langle Sx,x\rangle\vert\leq\Vert Sx\Vert\Vert x\Vert=1
\end{align*}
As Cauchy-Schwartz inequality is being achieved, we get that $\lambda$ is an eigenvalue of $S$ which contradicts $\sigma_e(S)=\emptyset$. So, $W(S)\subseteq\mathbb{D}$. Hence $W(S)=\mathbb{D}$.

Let $k\in\mathbb{N}_{\infty}$. Note that $\Omega_k(S)\subseteq\Omega_1(S)=W(S)=\mathbb{D}$. Let $\mu\in\mathbb{D}$. Suppose $H$ is a half closed-half plane at $\mu$. We claim that $\text{dim ran}E(H)\geq k$. If $0<\text{dim ran}E(H)<k$ then $S$ has eigenvalue which is a contradiction. If $\text{dim ran}E(H)=0$ then $\Omega_1(S)\subseteq H^c$ which contradicts that $\Omega_1(S)=\mathbb{D}$. So, $\text{dim ran}E(H)\geq k$. By Lemma \ref{equivalent criterion of being member of our set}, we obtain $\mu\in\Omega_k(S)$. So, $\Omega_k(S)=\mathbb{D}$. Hence by Theorem \ref{main theorem}, $\Lambda_k(S)=\mathbb{D}$ for $k\in\mathbb{N}_{\infty}$. 
\end{example}

However, the reader may look at Section $4$, \cite{RM} (Corollary 4.7, \cite{RM}) for a different proof of the description of the higher rank numerical range of the bilateral shift. The following example shows that $\infty$-rank numerical range could be empty.

\begin{example}\label{infinity rank numerical range being empty}
Consider $T=\bigoplus\limits_{n\geq 2}
    \begin{pmatrix}
    -\frac{1}{n} & 0 \\
    0 & \frac{e^{\frac{i\pi}{n}}}{n}
    \end{pmatrix}
$.  Then 
\begin{align*}
&\sigma_e(T)=\{-\frac{1}{n}:n\geq 2\}\cup\{\frac{e^{\frac{i\pi}{n}}}{n}:n\geq 2\},\\
&\sigma(T)=\sigma_e(T)\cup\{0\}.
\end{align*}
Let $\omega=\{z\in\mathbb{C}:\Im(z)<0\}$. Then $H_0=\omega\cup[0,\infty)$ is a half closed-half plane at $0$ with $\text{dim ran}E(H_0)=0$. So, $0\notin\Omega_1(T)$ and hence $0\notin\Omega_\infty(T)$ by Lemma \ref{relation between k and infinity}. Let $\lambda\neq 0$ be any point in $\mathbb{C}$. Then there exists a half closed-half plane $H_{\lambda}$ at $\lambda$ such that $\text{dim ran}E(H_{\lambda})<r$ for some $r\in\mathbb{N}$. So, by Lemma \ref{equivalent criterion of being member of our set}, $\lambda\notin\Omega_r(T)$ and hence $\lambda\notin\Omega_{\infty}(T)$ by Lemma \ref{relation between k and infinity}. Therefore, $\Omega_{\infty}(T)=\emptyset$. Hence by Theorem \ref{main theorem}, $\Lambda_{\infty}(T)=\Omega_{\infty}(T)=\emptyset$.

\end{example}

\begin{remark}\label{r}
In Example \ref{infinity rank numerical range being empty}, $V_{\infty}(T)=\{0\}$ (recalling $V_k(T)$ from Proposition \ref{nonemptyness}) since $0\in V_{\infty}(T)$ and for any $\lambda\neq 0$, there exists a half closed-half plane $H$ with $\lambda\in H^\circ$ such that $\text{dim ran}E(H)<r$ for some $r\in\mathbb{N}$ which implies $\lambda\notin V_r(T)$ and hence $\lambda\notin V_{\infty}(T)=\bigcap\limits_{k\geq 1}V_k(T)$.
\end{remark}

We now list down some consequences of Theorem \ref{main theorem}.

\begin{corollary}\label{compute numerical range}
Let $T\in\mathscr{B(\mathscr{H})}$ be a normal operator with no eigenvalue. Then $$\Lambda_k(T)=W(T),\quad \text{ for all } k\in\mathbb{N}_{\infty}.$$ 
\end{corollary}

\begin{proof}
Let $H\in\mathcal{S}_k$. Then $H$ is a half closed-half plane in $\mathbb{C}$ with $\text{dim ran}E(H)<k$. We want to show $H\in\mathcal{S}_1$. If possible let $\text{ dim ran }E(H)\geq1$. As $E(H)T=TE(H), \text{ ran }E(H)$ is a reducible subspace and $P_{\text{ran}E(H)}T|_{\text{ran}E(H)}$ is a normal matrix which implies $T$ has eigenvalue which contradicts our hypothesis. So, $\text{dim ran }E(H)<1$. Therefore, $H\in S_1$ and hence $\mathcal{S}_k\subseteq \mathcal{S}_1$. Clearly $\mathcal{S}_1\subseteq\mathcal{S}_k$ also. So, $\mathcal{S}_1=\mathcal{S}_k$. Now, by Theorem \ref{main theorem},
\begin{align*}
W(T)=\Lambda_1(T)=\Omega_1(T)=\bigcap\limits_{H\in S_1} H^c=\bigcap\limits_{H\in S_k} H^c=\Omega_k(T)=\Lambda_k(T).
\end{align*}   
This completes the proof.
\end{proof}

\begin{corollary}\label{necessary condition for infinity numerical range being empty}
Let $T\in\mathscr{B(\mathscr{H})}$ be normal with $\Lambda_{\infty}(T)=\emptyset$. Then $T$ is diagonalizable.
\end{corollary}
\begin{proof}
Let $T$ be not diagonalizable. Suppose  $T=T_1\oplus T_2$ with respect to the decomposition $\mathscr{H}=\text{ran }E(\sigma_e(T))\oplus\text{ran }E(\sigma_e(T)^c)$. Clearly $T_1, T_2$ are both normal and $T_2$ has no eigenvalue. As $\Lambda_{\infty}(T)=\emptyset$, we have $\Lambda_{\infty}(T_2)=\emptyset$. Since $T_2$ is normal with no eigenvalue, by Corollary \ref{compute numerical range}, we obtain $W(T_2)=\Lambda_{\infty}(T_2)$ which contradicts the fact that $W(T_2)$ is always non-empty. Hence $T$ is diagonalizable.
\end{proof}

\begin{remark}
Let $I$ be the identity operator acting on $\mathscr{H}$. Then $\Lambda_{k}(I)=W(I)=\{1\}$ for all $k\in\mathbb{N}_{\infty}$. Note that $I$ is a diagonalizable operator with the eigenvalue $1$. So, the converses of Corollary \ref{compute numerical range} and Corollary \ref{necessary condition for infinity numerical range being empty} are not necessarily true.
\end{remark}

Let us now give an alternative proof of Theorem 4.4, \cite{LPS}.

\begin{corollary}[Theorem 4.4, \cite{LPS}]\label{alternative proof}
Suppose $T\in\mathscr{B(\mathscr{H})}$ is normal and $1\leq k<\infty$. Then $\lambda\notin\Lambda_k(T)$ if and only if $T$ can be decomposed into $T_1\oplus T_2$ such that $T_1$ has dimension at most $k-1$, $W(T_1)\subseteq\lambda+S$ and $W(T_2)\subseteq\mathbb{C}\setminus(\lambda+S)$ where $S=e^{it}(\mathcal{P}\cup L)$ with $\mathcal{P}=\{z\in\mathbb{C}:\Im(z)>0\}$ and $L=(-\infty,0]$ or $[0,\infty)$ for some $t\in\mathbb{R}$.
\end{corollary}

\begin{proof}
`$\Rightarrow$'. Let $\lambda\notin\Lambda_k(T)$. Then by Theorem \ref{main theorem}, $\lambda\notin\Omega_k(T)$. By Lemma \ref{equivalent criterion of being member of our set}, there exists a half closed-half plane $H$ at $\lambda$ such that $\text{dim ran}E(H)<k$. Take $T_1=P_{\text {ran}E(H)}T|_{\text{ran}E(H)}$ and $T_2=P_{\text{ran}E(H^c)}T|_{\text{ran}E(H^c)}$. Then $T=T_1\oplus T_2$ such that dimension of $T_1$ is at most $k-1$ and $W(T_1)\subseteq H$ and $W(T_2)\subseteq H^c$.\\
`$\Leftarrow$'. Let $T$ can be decomposed into $T_1\oplus T_2$ such that $T_1$ has dimension at most $k-1$, $W(T_1)\subseteq\lambda+S$ and $W(T_2)\subseteq\mathbb{C}\setminus(\lambda+S)$ where $S=e^{it}(\mathcal{P}\cup L)$ with $\mathcal{P}=\{z\in\mathbb{C}:\Im(z)>0\}$ and $L=(-\infty,0]$ or $[0,\infty)$ for some $t\in\mathbb{R}$. Suppose $H=\lambda+S$. Then $H$ is a half closed-half plane at $\lambda$ with $\text{dim ran}E(H)<k$. So, by Lemma \ref{equivalent criterion of being member of our set}, $\lambda\notin\Omega_k(T)$. Hence by Theorem \ref{main theorem}, $\lambda\notin\Lambda_k(T)$.
\end{proof}

We end this section by the following observation.

\begin{corollary}
Let $T\in\mathscr{B(\mathscr{H})}$ be normal. Then $\lambda\notin\Lambda_\infty(T)$ if and only if $T$ can be decomposed into $T_1\oplus T_2$ such that $T_1$ has dimension at most finite, $W(T_1)\subseteq\lambda+S$ and $W(T_2)\subseteq\mathbb{C}\setminus(\lambda+S)$ where $S=e^{it}(\mathcal{P}\cup L)$ with $\mathcal{P}=\{z\in\mathbb{C}:\Im(z)>0\}$ and $L=(-\infty,0]$ or $[0,\infty)$ for some $t\in\mathbb{R}$.
\end{corollary}
\begin{proof}
It follows similarly as in the proof of Corollary \ref{alternative proof} using Theorem \ref{main theorem}.
\end{proof}

\section{Unitary dilations of a normal contraction}

In this section, we will see some interesting applications of the main theorem in this paper to the dilation theory. Let us first recall the definition of a dilation of an operator. Let $\mathscr{H}$ be a Hilbert space and $T\in \mathscr{B(\mathscr{H})}$. Suppose $\mathscr{K}$ is a Hilbert space containing $\mathscr{H}$. An operator $S\in\mathscr{B(\mathscr{K})}$ is said to be a dilation of $T$ (or $T$ is said to be a compression of $S$) if there exists a projection $P\in\mathscr{B(\mathscr{K})}$ on $\mathscr{H}$ such that $T=P_{\mathscr{H}}S|_{\mathscr{H}}$ or equivalently $S$ is unitarily similar to $2\times 2$ operator matrix $\begin{pmatrix}
    T & *\\
    * & *\\
  \end{pmatrix}$. Moreover, if $S$ being a dilation of $T$ be unitary then $S$ is said to be a unitary dilation of $T$ (cf. \cite{VP} for general background on the dilation theory).

Halmos \cite{PH} showed that every contraction $T\in\mathscr{B(\mathscr{H})}$ has a unitary dilation $U\in\mathscr{B(\mathscr{H}\oplus\mathscr{H})}$ of the following form $$U=\begin{pmatrix}
    T & -\sqrt{I-TT^*}\\
    \sqrt{I-T^*T} & T^*\\
  \end{pmatrix}.$$
It generated a lot of research including far reaching {Sz.-Nagy} dilation theorem regarding power unitary dilation of a contraction. Let $T\in \mathscr{B(\mathscr{H})}$ be a contraction. Halmos \cite{PH2} conjectured in 1962 that
\begin{align}\label{Halmos conjecture}
W(T)=\bigcap\left\{W(U): U \text{ is a unitary dilation of T}\right\}.
\end{align}
Durszt \cite{D} gave an example of a normal contraction to settle (\ref{Halmos conjecture}) in negative (1964). Later, Choi and Li have proved the following (2001).

\begin{theorem}[Theorem 2.4, \cite{CL}]\label{Halmos conjecture2}
Let $T$ be a contraction acting on a separable Hilbert space $\mathscr{H}$. Then
\begin{align*}
\overline{W(T)}=\bigcap\left\{\overline{W(U)}: U\in\mathscr{B}\mathcal{(H\oplus H)} \text{ is a unitary dilation of A}\right\}.
\end{align*}
\end{theorem}

Gau, Li and Wu generalised Theorem \ref{Halmos conjecture2} for the higher rank numerical range of a contraction acting on a finite dimensional Hilbert space.

\begin{theorem}(Theorem 1.2, \cite{GLW})\label{Halmos conjecture for HRNR on finite dimenion}
	Let $T\in\mathbb M_n$ be a contraction and $1\leq k\leq n$. Set $d_T=\mbox{dim~ran}(I-T^*T)^{\frac{1}{2}}.$ Then
	\begin{align*}
	\Lambda_k(T)=\bigcap\left\{\Lambda_k(U): U\in M_{n+d_T}  \text{ is a unitary dilation of T}\right\}.
	\end{align*}
\end{theorem}

We have further generalized Theorem \ref{Halmos conjecture for HRNR on finite dimenion} for the higher rank numerical range of a contraction acting on an infinite dimensional separable Hilbert space \cite{DM}.

\begin{theorem}[\cite{DM}]\label{Halmos conjecture for HRNR on infinite dimenion}
Let $T\in\mathscr{B(\mathscr{H})}$ be a contraction and $k\in\mathbb{N}_{\infty}$. Then
\begin{align*}
\overline{\Lambda_k(T)}=\bigcap\left\{\overline{\Lambda_k(U)}: U\in\mathscr{B}\mathcal{(H\oplus H)} \text{ is a unitary dilation of A}\right\}.
\end{align*}
\end{theorem}

The following theorem generalizes Theorem 2, \cite{D} by Durszt. It says that the closure sign on the higher rank numerical range in Theorem \ref{Halmos conjecture for HRNR on infinite dimenion} can not be omitted.

\begin{theorem}\label{generalising Durszt's example}
Let $k\in\mathbb{N}_{\infty}$. Then there exists a normal contraction $T$ for which the $k$-rank numerical range of all unitary dilations of $T$ contains $\Lambda_k(T)$ as a proper subset.
\end{theorem}
\begin{proof}
Let $T$ be a normal operator with $0$ being the only eigenvalue of $T$ with multiplicity $k$ and $\sigma(T)=\Gamma\cup\{0\}$ where $\Gamma=\{z\in\mathbb{C}:\Im{(z)}\geq 0,\; \vert z\vert=1\}$. Clearly $T$ is a contraction. We can construct such normal contraction by considering $T=T_1\oplus 0I_k$ where $T_1$ is the multiplication operator by the variable $z$ on the Hilbert space $L^2(\mu)$ where $\mu$ is the Lebesgue measure supported on $\Gamma$. Let $R=\{z\in\mathbb{C}:\Im{(z)}>0,\; \vert z\vert<1\}$. As $T$ does not have any eigenvalue on $\Gamma$ and $E_T$ has positive measure on any subarc of $\Gamma$, we get $R\cup\{0\}\subseteq\Lambda_k(T)$ by Corollary \ref{compute numerical range}. Now, by Lemma \ref{equivalent criterion of being member of our set} and Theorem \ref{main theorem}, we obtain $\Lambda_k(T)\subseteq R\cup\{0\}$. Therefore, $\Lambda_k(T)=R\cup\{0\}$.

Let $U$ be a unitary dilation of $T$. Then $R\cup\{0\}\subseteq\Lambda_k(U)$. We claim that $(-1,1)\subseteq\Lambda_k(U)$. If possible let there exists $r\in (-1,1)$ such that $r\notin\Lambda_k(U)$. Without loss of generality, let $r>0$. Again by Theorem \ref{main theorem} and Lemma \ref{equivalent criterion of being member of our set}, there exists a half closed-half plane $H_r$ at $r$ such that $\text{dim ran}E_U(H_r)<k$. As $R\subseteq\Lambda_k(U)$, $\delta H_r$ has to be the real axis. Again as $0\in\Lambda_k(U)$, the only possible choice of $H_r=\{z\in\mathbb{C}:\Im{(z)}<0\}\cup[r,\infty)$. Let $H_0=\{z\in\mathbb{C}:\Im{(z)}<0\}\cup[0,\infty)$. As $\text{dim ran}E_U(H_r)<k$, we have $\text{dim ran}E_U(H_0)<k$. So, by Lemma \ref{equivalent criterion of being member of our set}, $0\notin\Omega_k(U)$ and hence $0\notin\Lambda_k(U)$ by Theorem \ref{main theorem} which is a contradiction. Therefore, $(-1,1)\subseteq\Lambda_k(U)$. This completes the proof.
\end{proof}

In fact, more generally, we obtain the following.

\begin{proposition}\label{generalizing only if direction of Wu's theorem}
Let $k\in\mathbb{N}_{\infty}$ and $T$ be a normal strict contraction with
\begin{align*}
\Lambda_k(T)=\bigcap\left\{\Lambda_k(U): U \text{ is a unitary dilation of T}\right\}
\end{align*}
Then whenever $\Lambda_k(T)$ contains a point $\lambda\in\delta\Lambda_k(T)$ it also contains every open line segment $(\lambda,\lambda')$ in $\delta\Lambda_k(T)$.
\end{proposition}

\begin{proof}
Let $\Lambda_k(T)$ contains a point $\lambda\in\delta\Lambda_k(T)$. Suppose $(\lambda,\lambda')$ is any open line segment in $\delta\Lambda_k(T)$. Let $L$ be the line extending $\lambda,\lambda'$ and $H$ be the open half plane generated by $L$ such that $H\cap\Lambda_k(T)=\emptyset$. Suppose $U$ is a unitary dilation of $T$. Then $\Lambda_k(T)\subseteq\Lambda_k(U)$. We claim that $(\lambda,\lambda')\subseteq\Lambda_k(U)$.  Let there exists $\lambda_{\circ}\in (\lambda,\lambda')$ such that $\lambda_{\circ}\notin\Lambda_k(U)$. By Theorem \ref{main theorem} and Lemma \ref{equivalent criterion of being member of our set}, there exists a half closed-half plane $H_{\lambda_{\circ}}$ at $\lambda_{\circ}$ such that $\text{dim ran}E_U(H_{\lambda_{\circ}})<k$. We now argue that $\delta H_{\lambda_\circ}=L$. If $\delta H_{\lambda_\circ}\neq L$ then there exists a ball $B$ with centre at $\frac{\lambda_{\circ}+\lambda'}{2}$ (or $\frac{\lambda_{\circ}+\lambda}{2}$) such that $B\subseteq H_{\lambda_{\circ}}$. As $\Lambda_k(T)\subseteq\Lambda_k(U)\subseteq H_{\lambda_{\circ}}^c$, we get $B\cap\Lambda_k(T)=\emptyset$ which contradicts the fact that $\frac{\lambda_{\circ}+\lambda'}{2}$ (or $\frac{\lambda_{\circ}+\lambda}{2}$) is a boundary point of $\Lambda_k(T)$. So, we conclude that $\delta H_{\lambda_\circ}=L$. Again as $\lambda\in\Lambda_k(U)$, the only possible choice of $H_{\lambda_{\circ}}=H\cup l_{\lambda_{\circ}}$ where $l_{\lambda_{\circ}}$ is a ray with initial point $\lambda_{\circ}$ passing through $\lambda'$. Let $H_{\lambda}=H\cup l_{\lambda}$ where $l_\lambda$ is a ray with initial point $\lambda$ passing through $\lambda_{\circ}$. Then $H_{\lambda}$ is a half closed-half plane at $\lambda$ with $\text{dim ran}E_U(H_\lambda)<k$ as $\text{dim ran}E_U(H_{\lambda_{\circ}})<k$ and $\vert\lambda\vert<1$. So, by Lemma \ref{equivalent criterion of being member of our set} and Theorem \ref{main theorem}, we get $\lambda\notin\Lambda_k(U)$ which is a contradiction. Hence $(\lambda,\lambda')\subseteq\Lambda_k(U)$. As $U$ is arbitrary unitary dilation of $T$, we have $(\lambda,\lambda')\subseteq\Lambda_k(T)$ by hypothesis.
\end{proof}

The converse of Proposition \ref{generalizing only if direction of Wu's theorem} is not true for $k\geq 2$ as shown in the following example.

\begin{example}
Let $R=\left\{z\in\mathbb{C}:-\frac{1}{2}\leq\Re{(z)},\Im{(z)}\leq\frac{1}{2}\right\}$ and $\mu$ be the planer Lebesgue measure on $R$. Suppose $k\geq 2$. Let $T=T_1\oplus T_2$ where $T_1$ is the multiplication operator by the variable $z$ on the Hilbert space $L^2(\mu)$ and $T_2=(\frac{1}{2}+\frac{i}{4})I_{k-1}\oplus\left[\frac{1}{2}-\frac{i}{4}\right]\in M_k$. Clearly $T$ is a normal strict contraction with $\sigma_e(T)=\left\{\frac{1}{2}+\frac{i}{4},\frac{1}{2}-\frac{i}{4}\right\}$ and $\sigma(T)=R$. Now, by Lemma \ref{equivalent criterion of being member of our set} and Theorem \ref{main theorem}, we obtain
\begin{align*}
\Lambda_k(T)=\left\{z\in\mathbb{C}:-\frac{1}{2}<\Re{(z)},\Im{(z)}<\frac{1}{2}\right\}.
\end{align*}
Note that $\frac{1}{2}\notin\Lambda_k(T)$ and the condition whenever $\Lambda_k(T)$ contains a point $\lambda\in\delta \Lambda_k(T)$ it also contains every open line segment $(\lambda,\lambda')$ in $\delta \Lambda_k(T)$ is being satisfied here vacuously.

Let $U$ be a unitary dilation of $T$. Suppose $\frac{1}{2}\notin\Lambda_k(U)$. Then by Theorem \ref{main theorem} and Lemma \ref{equivalent criterion of being member of our set}, there exists a half closed-half plane $H$ at $\frac{1}{2}$ such that $\text{dim ran}E_U(H)<k$. As $\Lambda_k(T)\subseteq\Lambda_k(U)$, we can see $H$ is the union of $\{z\in\mathbb{C}:\Re{(z)}>\frac{1}{2}\}$ with either $L_1$ or $L_2$ where $L_1, L_2$ are two opposite directed rays with initial point $\frac{1}{2}$ on the line $\{z\in\mathbb{C}:\Re{(z)}=\frac{1}{2}\}$. Clearly $U=U_1\oplus U_2$ where $U_1$ is a unitary dilation of $T_1$ and $U_2$ is the same for $T_2$. As $\text{dim ran}E_U(H)<k$, we have $\text{dim ran}E_{U_2}(H)<k$. Denote $\Gamma_1=\mathbb{T}\cap H$ and $\Gamma_2=\mathbb{T}\cap H^c$. Let $U_2=U_2'\oplus U_2''$ with respect to the decomposition $\text{ran}E_{U_2}(\Gamma_1)\oplus\text{ran}E_{U_2}(\Gamma_2)$. As codimension of $\text{ran}E_{U_2}(\Gamma_2)$ is less than $k$ and $T_2\in M_k$, there exists $\xi\in\text{ran}E_{U_2}(\Gamma_2)\cap\mathbb{C}^k$ with $\Vert\xi\Vert=1$ such that $\langle T_2\xi,\xi\rangle=\langle U_2''\xi,\xi\rangle\in\text{conv}\Gamma_2$ which is a contradiction as $W(T_2)=\left[\frac{1}{2}-\frac{i}{4},\frac{1}{2}+\frac{i}{4}\right]$. So, $\frac{1}{2}\in\Lambda_k(U)$. Therefore, the higher rank numerical range of every unitary dilations of $T$ contains $\Lambda_k(T)$ as a proper subset. 
\end{example}

\begin{example}
Let $\Gamma=\{z\in\mathbb{C}:\Re{(z)}\leq 0,\; \vert z\vert=1\}$ and $\mu$ is a Lebesgue measure supported on $\Gamma$. Suppose $k\in\mathbb{N}_\infty$. Let $T=T_1\oplus iI_k$ where $T_1$ is the multiplication operator by the variable $z$ on the Hilbert space $L^2(\mu)$. Note that $\sigma_e(T)=\{i\}$ and $\sigma(T)=\Gamma$. By Lemma \ref{equivalent criterion of being member of our set}, Theorem \ref{main theorem} and Corollary \ref{compute numerical range}, we obtain
\begin{align*}
\Lambda_k(T)=\{z\in\mathbb{C}:\Re{(z)}<0,\; \vert z\vert<1\}\cup\{i\}.
\end{align*}
We have here $\Lambda_k(T)=\bigcap\left\{\Lambda_k(U): U \text{ is a unitary dilation of T}\right\}$ (as $T$ is unitary) but $\Lambda_k(T)$ does not contain $(i,-i)\subseteq\delta\Lambda_k(T)$ despite containing $i\in\delta\Lambda_k(T)$. Note here that $T$ is not a strict contraction.
\end{example}

The next theorem gives a necessary and sufficient condition on a normal contraction $T$ acting on an infinite dimensional separable Hilbert space that its higher rank numerical range equals the intersection of the higher rank numerical ranges of unitary dilations of $T$. This theorem strengthens and generalises Theorem 3.2, \cite{W} by Wu. The condition and the proof given there are erroneous as can be seen in the following example.

\begin{example}\label{counter example for Wu's theorem}
Consider $T=\bigoplus\limits_{n\geq 2}
    \begin{pmatrix}
    -\frac{1}{n} & 0 \\
    0 & \frac{e^{\frac{i\pi}{n}}}{n}
    \end{pmatrix}$. As $T$ is diagonalizable, it can be computed directly (or otherwise) that $W(T)=\text{conv}[\{-\frac{1}{n}:n\geq 2\}\cup\{\frac{e^{\frac{i\pi}{n}}}{n}:n\geq 2\}]$. Note that $0\notin W(T)$ and whenever $W(T)$ contains a point $\lambda\in\delta W(T)$ it also contains every open line segment $(\lambda,\lambda')$ in $\delta W(T)$.

Let $U$ be a unitary dilation of $T$.  Then $W(T)\subseteq W(U)$. Let $0\not\in W(U)$. Then by Theorem \ref{main theorem} and Lemma \ref{equivalent criterion of being member of our set}, there exists a half closed-half plane $H_0$ at $0$ such that $E_U(H_0)=0$. As $W(T)\subseteq W(U)$, $\delta H_0$ has to be the real axis. Again as $\{-\frac{1}{n}:n\geq 2\}\subseteq W(U)$, the only possible choice of $H_0=\{z\in\mathbb{C}:\Im{(z)}<0\}\cup [0,\infty)$. Then $H_{-\frac{1}{2}}=\{z\in\mathbb{C}:\Im{(z)}<0\}\cup [-\frac{1}{2},\infty)$ is a half closed-half plane at $-\frac{1}{2}$ with $E_U(H_{-\frac{1}{2}})=0$ as $E_U(H_0)=0$. Again by Lemma \ref{equivalent criterion of being member of our set} and Theorem \ref{main theorem}, we have $-\frac{1}{2}\notin W(U)$ which is a contradiction. Hence, $0\in W(U)$. Therefore, the numerical range of every unitary dilations of $T$ contains $W(T)$ as a proper subset.
\end{example}

The following is the main theorem of this section.

\begin{theorem}\label{generalizing Wu's theorem}
Let $T$ be a normal strict contraction and $k\in\mathbb{N}_{\infty}$. Then
\begin{align*}
\Lambda_k(T)=\bigcap\left\{\Lambda_k(U): U \text{ is a unitary dilation of T}\right\}
\end{align*}
holds if and only if for any $\lambda\in\delta\Lambda_k(T)\setminus \Lambda_k(T)$, there exists a closed half plane $H$ with the line $\delta H$ passing through $\lambda$ such that $\text{dim ran}E(H)<k$.
\end{theorem}

\begin{proof}
``$\Rightarrow$''. Suppose $\Lambda_k(T)=\bigcap\left\{\Lambda_k(U): U \text{ is a unitary dilation of T}\right\}$ holds. Let $\lambda\in\delta\Lambda_k(T)\setminus\Lambda_k(T)$. Then by hypothesis, there exists a unitary dilation $U$ of $T$ such that $\lambda\in\delta\Lambda_k(U)\setminus\Lambda_k(U)$. So, by Theorem \ref{main theorem} and Lemma \ref{equivalent criterion of being member of our set}, there exists a half closed-half plane $H_{\lambda}$ at $\lambda$ such that $\text{dim ran}E_U(H_\lambda)<k$. Denote $\Gamma_1=\mathbb{T}\cap H_\lambda$ and $\Gamma_2=\mathbb{T}\cap H_\lambda^c$. Let $U=U_1\oplus U_2$ with respect to the decomposition $\text{ran}E_U(\Gamma_1)\oplus\text{ran}E_U(\Gamma_2)$. Suppose $H=\overline{H_\lambda}$. Then $H$ is a closed half plane with the line $\delta H$ passing through $\lambda$. Let $T=T_1\oplus T_2$ with respect to the decomposition $\text{ran}E(H)\oplus\text{ran}E(H^c)$. Let $\text{dim ran}E(H)\geq k$. As codimension of $\text{ran}E_U(\Gamma_2)$ is less than $k$, there exists $\xi\in\text{ran}E_U(\Gamma_2)\cap\text{ran}E(H)$ with $\Vert\xi\Vert=1$ such that $\langle T_1\xi,\xi\rangle=\langle U_2\xi,\xi\rangle$. As $T_1$ is a strict contraction, we have $\langle T_1\xi,\xi\rangle=\langle U_2\xi,\xi\rangle\in H^c$ which contradicts that $W(T_1)\subseteq H$. So, $\text{dim ran}E(H)<k$.\\
``$\Leftarrow$''. It is clear that
\begin{align*}
\Lambda_k(T)\subseteq\bigcap\left\{\Lambda_k(U): U \text{ is a unitary dilation of T}\right\}.
\end{align*}
Let $\lambda\notin\Lambda_k(T)$. We need to show that there exists a unitary dilation $U$ of $T$ such that $\lambda\notin\Lambda_k(U)$. Suppose $\lambda\notin\overline{\Lambda_k(T)}$. Then by Theorem \ref{geomertic description for operator} and Theorem 5.2 \cite{LPS}, there exists $\alpha\in[0,2\pi)$ such that $\Re{(e^{i\alpha}\lambda)}>\lambda_k(\Re{(e^{i\alpha}T)})$. We choose
\begin{align*}
U=\begin{pmatrix}
    T & -e^{-i\alpha}\sqrt{I-TT^*}\\
    e^{-i\alpha}\sqrt{I-T^*T} & e^{-2i\alpha}T^*\\
  \end{pmatrix}.
\end{align*}
Then $U$ is a unitary dilation of $T$ with $\Re{(e^{i\alpha}\lambda)}>\lambda_k(\Re{(e^{i\alpha}T)})=\lambda_k(\Re{(e^{i\alpha}U)})$. So, again by Theorem \ref{geomertic description for operator} and Theorem 5.2, \cite{LPS}, we obtain $\lambda\notin\overline{\Lambda_k(U)}$ and hence $\lambda\notin\Lambda_k(U)$.

Next, let $\lambda\in\delta\Lambda_k(T)\setminus\Lambda_k(T)$. Then by hypothesis, there exists a closed half plane $H$ with the line $\delta H$ passing through $\lambda$ such that $\text{dim ran}E(H)<k$. Let $T=T_1\oplus T_2$ with respect to the decomposition $\text{ran}E(H)\oplus\text{ran}E(H^c)$. Clearly $T_1, T_2$ are both normal. Without loss of generality, we can take $T_1=\text{diag}(d_1,d_2,\cdots,d_r)$ for some $r<k$. Denote $\Gamma_1=\mathbb{T}\cap H$ and $\Gamma_2=\mathbb{T}\cap H^c$. Let $j\in\{1,2,\cdots,r\}$. Choose $\xi_j\in\Gamma_1$ and $\eta_j\in\Gamma_2$ such that $d_j$ lies on the line segment joining $\xi_j$ and $\eta_j$. Then $V_j=\text{diag}(\xi_j,\eta_j)$ is a unitary dilation of $(d_j)$. So, $U_1=\sum^\oplus_jV_j$ is a unitary dilation of $T_1$ with $\text{dim ran}E_{U_1}(H)=r<k$. Note that $W(T_2)\subseteq H^c$. Then there exists $\beta\in[0,2\pi)$ and $\mu\in\mathbb{R}$ such that $W(e^{i\beta}T_2)\subseteq\{z\in\mathbb{C}:\Re{(z)}<\mu\}$. We again choose
\begin{align*}
U_2=\begin{pmatrix}
    T_2 & -e^{-i\beta}\sqrt{I-T_2T_2^*}\\
    e^{-i\beta}\sqrt{I-T_2^*T_2} & e^{-2i\beta}T_2^*\\
  \end{pmatrix}.
\end{align*}
Then $U_2$ is a unitary dilation of $T_2$ with $W(e^{i\beta}U_2)\subseteq\{z\in\mathbb{C}:\Re{(z)}<\mu\}$ which implies $W(U_2)\subseteq H^c$. So, $\text{dim ran}E_{U_2}(H)=0$. Then $U=U_1\oplus U_2$ is a unitary dilation of $T$ with $\text{dim ran}E_U(H)=r<k$. Hence by Lemma \ref{equivalent criterion of being member of our set} and Theorem \ref{main theorem}, we obtain $\lambda\notin\Lambda_k(U)$. 
\end{proof}

In the following conjecture, we propose a necessary and sufficient condition on a contraction $T$ acting on an infinite dimensional separable Hilbert space that its higher rank numerical range equals the intersection of the higher rank numerical ranges of unitary dilations of $T$ which, in particular, gives a necessary and sufficient condition for the Halmos conjecture (\ref{Halmos conjecture2}) to be true.

\begin{conjecture}\label{NASC for Halmos conjecture to be true}
Let $T$ be a strict contraction and $k\in\mathbb{N}_{\infty}$. Then
\begin{align*}
\Lambda_k(T)=\bigcap\left\{\Lambda_k(U): U \text{ is a unitary dilation of T}\right\}
\end{align*}
holds if and only if for any $\lambda\in\delta\Lambda_k(T)\setminus\Lambda_k(T)$, there exists $\theta\in[0,2\pi)$ such that
\begin{align*}
\text{dim ran}E_S([0,\infty))<k \text{ where } S=\Re{(e^{i\theta}T-\lambda)}.
\end{align*}
\end{conjecture}

We will see that the condition given in the Conjecture \ref{NASC for Halmos conjecture to be true} is necessary as discussed below. However, we strongly believe that the given condition is also sufficient which we wish to investigate in future.

\begin{proof}[\textbf{Proof of ``$\Rightarrow$'' in Conjecture \ref{NASC for Halmos conjecture to be true}}]
Suppose
\begin{align*}
\Lambda_k(T)=\bigcap\left\{\Lambda_k(U): U \text{ is a unitary dilation of T}\right\}
\end{align*} 
holds. Let $\lambda\in\delta\Lambda_k(T)\setminus\Lambda_k(T)$. Then by hypothesis, there exists a unitary dilation $U$ of $T$ such that $\lambda\in\delta\Lambda_k(U)\setminus\Lambda_k(U)$. Now, by Theorem \ref{main theorem} and Lemma \ref{equivalent criterion of being member of our set}, there exists a half closed-half plane $H_\lambda$ at $\lambda$ such that $\text{dim ran}E_U(H_\lambda)<k$. Set $\Gamma_1=\mathbb{T}\cap H_\lambda$ and $\Gamma_2=\mathbb{T}\cap H_\lambda^c$. Let $U=U_1\oplus U_2$ with respect to the decomposition $\text{ran}E_U(\Gamma_1)\oplus\text{ran}E_U(\Gamma_2)$. Suppose $H=\overline{H_\lambda}$. Then $H$ is a closed half plane with the line $\delta H$ passing through $\lambda$. Then there exists $\theta\in[0,2\pi)$ such that $H'=e^{i\theta}H-\lambda=\{z\in\mathbb{C}:\Re{(z)\geq 0}\}$. Denote $\Gamma_1'=e^{i\theta}\Gamma_1-\lambda, \Gamma_2'=e^{i\theta}\Gamma_2-\lambda, T'=e^{i\theta}T-\lambda, U'=e^{i\theta}U-\lambda$ and $U'=U_1'\oplus U_2'$ with respect to the decomposition $\text{ran}E_{U'}(\Gamma_1')\oplus\text{ran}E_{U'}(\Gamma_2')$. Let $\text{dim ran}E_S([0,\infty))\geq k$. As codimension of $\text{ran}E_{U'}(\Gamma_2')$ is less than $k$, there exists $\xi\in\text{ran}E_{U'}(\Gamma_2')\cap\text{ran}E_S([0,\infty))$ with $\Vert\xi\Vert=1$ such that $\langle T'\xi,\xi\rangle=\langle U_2'\xi,\xi\rangle$. So, $\langle T\xi,\xi\rangle=\langle U_2\xi,\xi\rangle$. As $T$ is a strict contraction, we have $\langle T\xi,\xi\rangle=\langle U_2\xi,\xi\rangle\in H^c$ which implies $\langle S \xi,\xi\rangle=\langle\Re{(U_2')}\xi,\xi\rangle<0$. As $\xi\in\text{ran}E_S([0,\infty))$, we also have
\begin{align*}
\langle S\xi,\xi\rangle=\int_\mathbb{R}x \; d\langle E_S(.)\xi,\xi\rangle=\int_0^\infty x \; d\langle E_S(.)\xi,\xi\rangle=\int_0^\infty x \; d\Vert E_S(.)\xi\Vert^2\geq 0
\end{align*} 
which leads to a contradiction. Hence, $\text{dim ran}E_S([0,\infty))<k$.
\end{proof}

We end this section by posing the following question which we could not able to answer at this stage. In fact, the answer is not known even for $k=1$.

\begin{question}
Let $T\in\mathscr{B(\mathscr{H})}$ and $k\in\mathbb{N}_{\infty}$. Is
\begin{align*}
\Lambda_k(T)=\bigcap\left\{\Lambda_k(N): N \text{ is a normal dilation of T}\right\}?
\end{align*}
\end{question}

\section*{Acknowledgements}
The research of the first author is supported by the Integrated Ph.D fellowship of Indian Institute of Science Education and Research Thiruvananthapuram.

\bibliographystyle{elsarticle-num}

\end{document}